\documentclass[12pt, reqno]{amsart}
\usepackage{amsmath, amsthm, amscd, amsfonts, amssymb, graphicx, color}
\usepackage[bookmarksnumbered, colorlinks, plainpages]{hyperref}
\hypersetup{colorlinks=true,linkcolor=red,
anchorcolor=green, citecolor=cyan, urlcolor=red, filecolor=magenta, pdftoolbar=true}

\allowdisplaybreaks 

\newtheorem{theorem}{Theorem}[section]

\newtheorem{proposition}[theorem]{Proposition}

\theoremstyle{definition}
\newtheorem{definition}[theorem]{Definition}

\numberwithin{equation}{section}

\begin{document}

\setcounter{page}{1}

\title[Generalized g-Bessel multipliers]{Invertibility of generalized g-frame multipliers\\ in Hilbert spaces}

\author[M. Abolghasemi, Y. Tolooei \MakeLowercase{and} Z. Moosavianfard]{M. Abolghasemi, Y. Tolooei \MakeLowercase{and} Z. Moosavianfard}

\address{Department of Mathematics, Faculty of Sciences, Razi University, Kermanshah, Iran.}
\email{\textcolor[rgb]{0.00,0.00,0.84}{m$\_$abolghasemi@razi.ac.ir , Y. Toloei@razi.ac.ir;
moosaviyanfard.zahra66@yahoo.com}}



\begin{abstract}
In this paper, we investigate the invertibility of generalized g-Bessel multipliers.
We show that for semi-normalized symbols,
the inverse of any invertible generalized g-frame multiplier can be represented as
a generalized g-frame multiplier. Also we give several approaches
for constructing invertible generalized g-frame multipliers from the given one.
It is worth mentioning that some of our results are quite different
from those studied in the previous literatures on this topic.\\\\
\footnotesize\emph{\emph{2010} MSC}: 42C15, 47A05, 41A58.
\end{abstract}
\addtolength{\baselineskip}{1mm} \maketitle

\section*{Introduction}

In recent years there has been shown considerable interest by functional analysts
in the study of Bessel multipliers as a generalization of the frame operators, approximately dual frames \cite{japp}, generalized dual frames
\cite[Remark 2.8(ii)]{japp} and atomic systems for subspaces \cite{jatom}. In fact, the study of this class of operators leads us to new results concerning dual frames and
local atoms, two concepts at the core of frame theory.

The notions Fourier and Gabor
multipliers were extended to ordinary Bessel multipliers in Hilbert spaces by Balazs \cite{balaz3},
$p$-Bessel sequences in Banach spaces by Balazs and Rahimi in \cite{Rahi}, von Neumann-Schatten setting \cite{JLAA} and continuous setting in \cite{continu}. In \cite{Stoeva}, sufficient and/or necessary conditions for invertibility of ordinary Bessel multipliers have determined depending on
the properties of the analysis and synthesis sequences, as well as the symbol. Later on, in \cite{balaz11}, Stoeva and Balazs have considered the representation of the inverse of an ordinary frame multiplier. Moreover, the invertibility of Bessel multipliers in a much more general setting has been considered by Javanshiri
and his coauthor in \cite{JNF,JLAA}.

On the other hand, Rahimi \cite{rahim1} introduced
and studied the concept of Bessel multipliers for $g$-Bessel sequences in Hilbert spaces.
Recall that, $g$-Bessel sequences as an interesting generalization of ordinary Bessel sequence were first considered by Sun \cite{Sun,Sun2}.
It seems to the author that the invertibility of $g$-Bessel multipliers has not been touched so far. The reader will remark that, $g$-frames are quite different from ordinary frames; For example,
an exact $g$-frame in a Hilbert space is not equivalent to a $g$-Riesz basis, whereas
an exact frame is equivalent to a Riesz basis. This guarantees that
the study of $g$-frames and other related
concepts is more complicated than that of ordinary frames in Hilbert spaces.

Our purpose here is to consider the representation of the inverse of an invertible $g$-frame multiplier. For this purpose, we discuss a new result about the dual of $g$-frames. Moreover, in the case where the symbols is semi-normalized, we show that the inverse of any invertible $g$-frame multiplier can always be represented as a $g$-frame multiplier with the reciprocal symbol and dual $g$-frames of the given ones.
Finally, we investigate the matrix representation as well as the diagonalization of operators
on a Hilbert space $\mathcal U$ with $g$-frames.


\section{preliminaries}

In this section we have collected some notations
and results which are needed for the subsequent sections.
Throughout the paper $\mathcal{H}$ and $\mathcal K$ are separable Hilbert spaces; $\mathit{I}$ is a subset of $\mathbb{Z}$ and $\{\mathcal{K}_{i}\}_{i\in I}$ is a sequence of closed subspaces of $\mathcal K$. The notation $B( \mathcal{H}, \mathcal{K}_{i})$
denotes the Banach space of all bounded linear operators from $\mathcal{H}$ into $\mathcal{K}_{i}$; $Id_{\mathcal H}$ denotes the identity operator on $\mathcal H$; $\Lambda$ and $\Gamma$ are used to denote the sequences $\{\Lambda_i\}_{i\in I}$ and $\{\Gamma_i\}_{i\in I}$ with elements from $B({\mathcal H},{\mathcal K}_i)$, respectively. Moreover, we assume that  $\ell^2(\oplus_{i\in I}{\mathcal K}_i)$ is the Hilbert space
$$\ell^2(\oplus_{i\in I}{\mathcal K}_i)={\Big\{}\{x_i\}_{i\in I}:~\forall i~ x_i\in {\mathcal K}_i~ {\hbox{and}}~\sum_{i\in I}\|x_i\|^2<\infty{\Big\}},$$ with the inner product given by
$\big<\{x_i\}_{i\in I},\{y_i\}_{i\in I}\big>=\sum_{i\in I}\big<x_i,y_i\big>.$

Now, let us recall from \cite{} the definition of g-frame which includes the ordinary
frames and many recent generalizations of ordinary frames.

\begin{definition}
A sequence $\Lambda$ is called a generalized frame or simply
a g-frame for $\mathcal H$ with respect to
$\{{\mathcal K}_i\}_{i\in I}$ if there are two positive constants $A_{\Lambda}$ and $B_{\Lambda}$ such that
\begin{equation}\label{0839}
A_{\Lambda}\|x\|^2\leq\sum_{i\in I}\|\Lambda_ix\|^2\leq B_{\Lambda}\|x\|^2,
\quad\quad\quad (x\in{\mathcal H}).
\end{equation}
We call $A_{\Lambda}$ and $B_{\Lambda}$ the lower and upper frame bounds, respectively. In particular,
the sequence $\Lambda$ is called a $g$-Bessel
sequence, if only in {\rm(\ref{0839})} the second inequality holds.
\end{definition}

If $\Lambda$ is a $g$-Bessel sequence for $\mathcal H$ with respect to
$\{{\mathcal K}_i\}_{i\in I}$, then
$$T_{\Lambda}:\ell^2(\oplus_{i\in I}{\mathcal K}_i)\rightarrow{\mathcal H};\quad\{x_i\}_{i\in I}\mapsto\sum_{i\in I}\Lambda^*_i(x_i)$$ denote the associated synthesis operator. Its adjoint $T_{\Lambda}^*$ is called the analysis operator of $\Lambda$ which can be obtained as follows
$$T_{\Lambda}^*:{\mathcal H}\rightarrow\ell^2(\oplus_{i\in I}{\mathcal K}_i);\quad x\mapsto\{\Lambda_i(x)\}_{i\in I}.$$
It is proved that $\Lambda$ is a g-frame if and only if $T_{\Lambda}$ is a bounded operator which maps $\ell^2(\oplus_{i\in I}{\mathcal K}_i)$ surjectively onto $\mathcal H$. In particular, if $\Lambda$ is a g-frame, its frame operator given by
$$S_{\Lambda}(x):=T_{\Lambda}T_{\Lambda}^*(x)=\sum_{i\in I}\Lambda^*_i\Lambda_i(x)\quad\quad\quad(x\in{\mathcal H}),$$
is a bounded and positive self-adjoint operator in $B(\mathcal H)$, the Banach space of all bounded operators from $\mathcal H$ into $\mathcal H$.
This leads to the following
reconstruction formula
$$x=\sum_{i\in I}\Lambda^*_i\Lambda_iS_{\Lambda}^{-1}(x)=\sum_{i\in I}(\Lambda_iS_{\Lambda}^{-1})^*\Lambda_i(x),$$
for all $x\in{\mathcal H}$. As usual, the sequence $\widetilde{\Lambda}:=\{\Lambda_iS_{\Lambda}^{-1}\}_{i\in I}$ is called the canonical dual g-frame of $\Lambda$ which is a g-frame for $\mathcal{H} $ with frame operator $S_{\Lambda}^{-1} $ and frame bounds $B_{\Lambda}^{-1} $ and $A_{\Lambda}^{-1}$.

Finally, we recall from \cite{gene} that for two g-Bessel sequences $\Lambda$ and $\Gamma$ and a bounded operator $U:\ell^2(\oplus_{i\in I}{\mathcal K}_i)\longrightarrow \ell^2(\oplus_{i\in I}{\mathcal K}_i)$, the operator $M_{U,\Lambda,\Gamma}: \mathcal{H} \longrightarrow \mathcal{H}$ defined by
\begin{equation*}
M_{U,\Lambda,\Gamma}:=T_{\Lambda}U T^{*}_{\Gamma},
\end{equation*}
is called generalized multiplier of g-Bessel sequences $\Lambda $ and $\Gamma $ with symbol $U$.
Particularly, if for every $j\in I$, we define the operators $\pi_j$ and $\iota_j$ as
$$\pi_j:\ell^2(\oplus_{i\in I}{\mathcal K}_i)\rightarrow{\mathcal H}_j;\quad \{x_i\}_{i\in I}\mapsto x_j,$$
and
$$\iota_j:{\mathcal H}_j\rightarrow\ell^2(\oplus_{i\in I}{\mathcal K}_i);\quad x_j\mapsto\{\delta_{i,j}x_j\}_{i\in I},$$
where $\delta_{i,j}$ denotes the Kronecker delta, then the operator $M_{U,\Lambda,\Gamma}$ enjoys the following representation
\begin{equation}\label{parsa0900}
M_{U,\Lambda,\Gamma}(x)=\sum_{i\in I}\sum_{j\in I}{\Lambda}_i^*u_{ij}{\Gamma}_j(x),
\end{equation}
where $u_{ij}\in B(\mathcal{H}_{j}, \mathcal{H}_{i})$
defined by $u_{ij}(x_j)=\pi_iU\iota_j(x_j)$ and $U=[u_{ij}]$ is its matrix description.
We observe that, if one restricts the set of diagonal operators $U=diag\{u_i\}_{i\in I}$ with $u_i\in B(\mathcal{H}_{i}, \mathcal{H}_{i})$, then
formula (\ref{parsa0900}) becomes considerably simpler
\begin{equation}\label{parsa0910}
M_{U,\Lambda,\Gamma}(x)=\sum_{i\in I}{\Lambda}_i^*u_{i}{\Gamma}_i(x).
\end{equation}
Moreover, if for a weight $m=\{m_i\}_{i\in}\in\ell^\infty(I)$ we consider $u_i:{\mathcal H}_i\rightarrow{\mathcal H}_i$ by $u_i(x_i)=m_ix_i$, then (\ref{parsa0910}) reduces to
$$M_{m,\Lambda,\Gamma}(x)=\sum_{i\in I}m_i{\Lambda}_i^*{\Gamma}_i(x),$$
which has been studied by Rahimi \cite{rahim1} and in a much more general setting by Javanshiri and Choubin in \cite{JLAA}, where, here, and in the sequel $\ell^\infty(I)$ has its usual meanings.


\section{Some basic results on invertibility}

We commence this section by a discussion of why the invertibility of multipliers with the form Eq. (\ref{parsa0910}) is the main object of study of this paper.
To this end, first let us to note that, on the one hand, it is not hard to check that
the satisfying of g-Bessel sequences $\Lambda$ and $\Gamma$ in the lower g-frame condition are necessary for the invertibility of a generalized multiplier $M_{U,\Lambda,\Gamma}$ of the form Eq. (\ref{parsa0900}).
On the other hand, for given g-frames $\Lambda$ and $\Gamma$ there exists always
infinitely many non-injective operators $U\in B(\ell^2(\oplus_{i\in I}{\mathcal K}_i))$ such that the generalized multiplier $M_{U,\Lambda,\Gamma}=T_{\Lambda}U T^{*}_{\Gamma}$
is invertible whereas the injectivity of the operator $${\mathcal M}_m:\ell^2(I)\rightarrow\ell^2(I);\quad \{c_i\}_{i\in I}\mapsto\{m_ic_i\}_{i\in I},$$
was a very useful tool in the study of invertible ordinary Bessel multipliers, see for example \cite{}. Indeed, it suffices to set
$$U:=T_\Lambda^*S_\Lambda^{-1}S_{\Gamma}^{-1}T_\Gamma+P_{\ker(T_\Lambda)}R,$$
where $R$ is an arbitrary operator in $B(\ell^2(\oplus_{i\in I}{\mathcal K}_i))$.
This shows that there is too much freedom in the choice of the operator $U$ in
Eq. (\ref{parsa0900}) and it seems reasonable to work with particular classes of multipliers of the form Eq. (\ref{parsa0910}). Hence, in what follows $U$ refers to an operator in $B(\ell^2(\oplus_{i\in I}{\mathcal K}_i))$ which has the matrix description defined by $diag\{u_i\}_{i\in I}$ and $u_i\in B({\mathcal K}_i)$ for each $i\in I$.
Moreover, the letter semi-normalized is used for $U$ whenever
in addition to the invertibility of each $u_i$ ($i\in I$) the operator $D_U$ is also boundedly invertible, that is, the operator
$$D_{U^{-1}}:\ell^2(\oplus_{i\in I}{\mathcal K}_i)\rightarrow\ell^2(\oplus_{i\in I}{\mathcal K}_i);\quad \{x_i\}_{i\in I}\mapsto \{u_i^{-1}x_i\}_{i\in I}$$
is in $B(\ell^2(\oplus_{i\in I}{\mathcal K}_i))$.
It is worth mentioning that if $M_{U,\Lambda,\Gamma}$ is invertible for some g-Bessel sequence $\Gamma$, then routine calculations show that the g-Bessel sequences $\Lambda$, $\Gamma$, $U\Gamma:=\{u_i\Gamma_i\}_{i\in I}$ and $U\Lambda:=\{u_i^*\Lambda_i\}_{i\in I}$ must satisfy in the lower g-frame condition.


Our starting point is the following result which for fixed g-Bessel sequence $\Lambda$ and symbol $U$ characterizes all possible g-Bessel sequence $\Gamma$ that participate to construct invertible generalized multiplier $M_{U,\Lambda,\Gamma}$.

\begin{proposition}\label{minimal}
Let $\Lambda$ be a g-Bessel sequence for $\mathcal{H}$  with respect to $(\mathcal{K}_{i})_{i\in I}$ and let $U$ be a bounded operator on $\ell^2(\oplus_{i\in I}{\mathcal K}_i)$. The following assertions hold.
\begin{enumerate}
\item The g-Bessel sequences that participate to construct
invertible generalized multipliers with g-Bessel sequence $\Lambda$ and symbol $U$ are precisely the sequence $\Gamma$ satisfying
\begin{equation}\label{p0732}
u_i\Gamma_{i}=\Lambda_{i} S^{-1}_{\Lambda}T+\pi_{i}\Phi\quad\quad\quad\quad (i\in I)
\end{equation}
where $\Phi:\mathcal{H}\longrightarrow \ell^{2}(\oplus_{i}\mathcal{K}_{i})$ is an operator such that $T_{U}\Phi=0$ and $T$ is an invertible operator in $B(\mathcal{H})$.
\item The g-Bessel sequence $\Gamma$ participates to construct
invertible generalized multipliers with g-frame $\Lambda$ and symbol $U$
for which the analysis operator of $U\Gamma$ obtains the minimal norm if and only if
    $u_i\Gamma_{i}=\Lambda_{i} S^{-1}_{\Lambda}T$ ($i\in I$)
for some invertible operator $T$ in $B(\mathcal{H})$.
\end{enumerate}
\end{proposition}
\begin{proof}
The backward implication of (1) being trivial, we give the proof of the direct implication
only. To this end, suppose that $\Gamma$ is a g-Bessel sequence such that $M_{U,\Lambda,\Gamma} $ is invertible. Put $T= M_{U,\Lambda,\Gamma}$ and define the operator $\Phi:\mathcal{H}\longrightarrow  \ell^{2}(\oplus_{i}\mathcal{K}_{i} )$ by
$$\Phi=UT^{*}_{\Gamma}-T^{*}_{\Lambda}S^{-1}_{\Lambda}T$$
Then we observe that
$$T_{\Lambda}\Phi=T_{\Lambda}UT^{*}_{\Gamma}-T_{\Lambda}
T^{*}_{\Lambda}S^{-1}_{\Lambda}T=0,$$
and this completes the proof of (1).

In order to prove (2) it suffices to show that for any g-frame $\Gamma$ which satisfies in \ref{p0732} we have
$$\|T_{U\Gamma}^*\|^2\geq\Big(\widetilde{A}_\Lambda\|{T}^{-1}\|^2\Big)^{-1},$$
and $U\Gamma_o:=\{\Lambda_{i} S^{-1}_{\Lambda}T\}_{i\in I}$ is the unique g-frame for which $$\|T_{U\Gamma_0}^*\|^2=\Big(\widetilde{A}_{\Lambda}\|T^{-1}\|^2\Big)^{-1},$$
where $T= M_{U,\Lambda,\Gamma}$.
To this end, by definition, we observe that
$$\|x\|^2\leq\frac{1}{A_\Lambda}\|T_{\Lambda}^*x\|^2\quad\quad\quad(x\in{\mathcal H}).$$
It follows that
\begin{equation}\label{0157}
\frac{1}{\widetilde{A}_\Gamma}=\inf\Big\{A:~\|x\|^2\leq A\|T_{\Lambda}^*x\|^2\quad\forall~x\in{\mathcal H}\Big\}.
\end{equation}
On the other hand, we have
\begin{align*}
\|T^*x\|^2&=\big<TT^*x,x\big>\\
&=\big<T_{\Lambda}T_{U\Gamma}^*T_{U\Gamma}T_{\Lambda}^*x,x\big>\\
&\leq\|T_{U\Gamma}^*\|^2\big<T_{\Lambda}T_{\Lambda}^*x,x\big>\\
&=\|T_{U\Gamma}^*\|^2\|T_{\Lambda}^*x\|^2.
\end{align*}
From this, by equality
$\|x\|\leq\|T^{-1}\|\|T^*x\|$, we deduce that
$$\|x\|^2\leq\|T^{-1}\|^2\|T_{U\Gamma}^*\|^2\|T_{\Lambda}(x)\|^2.$$
This together with (\ref{0157}) implies that
$$\|T^{-1}\|^2\|T_{U\Gamma}^*\|^2\geq\frac{1}{\widetilde{A}_\Lambda}.$$
In order to prove that $\Gamma_0$ is the unique g-frame for which $$\|T_{U\Gamma_0}^*\|^2=\Big(\widetilde{A}_\Lambda\|T^{-1}\|^2\Big)^{-1},$$
we first make use of Douglas' Theorem for surjective operators $T_{\Lambda}$ and $T$ and find that there exists a unique operator $R:{\mathcal H}\rightarrow\ell^{2}(\oplus_{i}\mathcal{K}_{i})$ of minimal norm for which $T=T_\Lambda R$, particularly, we have
$$\|R\|^2=\inf\Big\{B:~\|T^*x\|^2\leq B\|Rx\|^2\quad\forall~x\in{\mathcal H}\Big\}.$$
On the other hand, an argument similar to the proof of \cite[Lemma 2.1]{Sun2} shows that
if $x$ has a representation $Tx=\sum_{i\in I}\Lambda^*_i x_i$ for some sequence $\{x_i\}_{i\in I}\in \ell^{2}(\oplus_{i}\mathcal{K}_{i})$, then
$$\sum_{i\in I}\|x_i\|^2=\sum_{i\in I}\|\Lambda_iS_\Lambda^{-1}T\|^2+\sum_{i\in I}\|x_i-\Lambda_iS_\Lambda^{-1}T\|^2.$$
It follows that $R=T_{U\Gamma_0}^*$ and thus
\begin{align*}
\|T_{U\Gamma_0}^*\|^2&=\inf\Big\{B:~\|T^*x\|^2\leq B\|T_{\Lambda}^*x\|^2\quad\forall~x\in{\mathcal H}\Big\}\\
&=\inf\Big\{B:~\|x\|^2\leq B\|T^{-1}\|^2\|T_{\Lambda}^*x\|^2\quad\forall~x\in{\mathcal H}\Big\}\\
&=\frac{1}{\|T^{-1}\|^2}\inf\Big\{A:~\|x\|^2\leq A\|T_{\Lambda}^*x\|^2\quad\forall~x\in{\mathcal H}\Big\}\\
&=\frac{1}{\widetilde{A}_\Lambda\|T^{-1}\|^2}.
\end{align*}
We have now completed the proof of the proposition.
\end{proof}


Next we turn our attention to the characterization of g-frames $\Lambda$ that
participate to construct invertible generalized multiplier $M_{U,\Lambda,\Gamma}$ for given g-frame $\Gamma$ and certain symbol $U$. Here it should be noted that the class of symbol $U$ satisfying the property of the next result is quite rich. It
contains for instance all positive semi-normalized sequence $\{u_i\}_{i\in I}\subset B({\mathcal K}_i)$ and positive semi-normalized scaler sequence $\{m_i\}_{i\in I}\subset (0,\infty)$ as well.

\begin{proposition}
Let $\Gamma$ be a g-frame for $\mathcal{H}$ with respect to $\{\mathcal{K}_{i}\}_{i\in I}$
and let $U$ be a bounded operator on $\ell^2(\oplus_{i\in I}{\mathcal K}_i)$.
Assume also that the sequence $V=\{v_i\}_{i\in I}\subset B({\mathcal K}_i)$ is such that $u_i=v_i^*v_i$ ($i\in I$) and the sequence $V\Gamma:=\{v_i\Gamma_i\}_{i\in I}$ is a g-frame. Then $M_{U,\Gamma,\Gamma}$ is invertible and particularly, the g-frame $\Lambda$ participates to construct
invertible g-Bessel multiplier $M_{U,\Lambda,\Gamma}$ if and only if there is an operator  $\Psi:\ell^{2}(\oplus_{i}\mathcal{K}_{i})\longrightarrow\mathcal{H}$ and invertible operators $T_1, T_2\in B(\mathcal{H})$ such that
$$T_{\Lambda}=T_1 \Big(M^{-1}_{U,\Gamma,\Gamma}T_{\Gamma}+T_2^{-1}\Psi(Id_{\ell^{2}(\oplus_{i}\mathcal{K}_{i} )}-UT^{*}_{\Gamma}M^{-1}_{U,\Gamma,\Gamma}T_{\Gamma})\Big).$$
\end{proposition}
\begin{proof}
That $M_{U,\Gamma,\Gamma}$ is invertible follows from the fact that it equal to the frame operator of g-frame $V\Gamma$. Now, suppose that $M_{U,\Lambda,\Gamma} $ is invertible and take $\Psi=T_{\Lambda}$ and $T_1=T_2= M_{U,\Lambda,\Gamma}$, then we observe that
\begin{align*}
T_{\Lambda}&=M_{U,\Lambda,\Gamma}M_{U,\Lambda,\Gamma}^{-1}T_{\Lambda}\\
&=M_{U,\Lambda,\Gamma}\Big(M^{-1}_{U,\Gamma,\Gamma}T_{\Gamma}+M_{U,\Lambda, \Gamma}^{-1}T_{\Lambda}-M^{-1}_{U,\Gamma,\Gamma}T_{\Gamma}\Big)\\
&=M_{U,\Lambda,\Gamma}\Big(M^{-1}_{U,\Gamma,\Gamma}T_{\Gamma}+M_{U,\Lambda, \Gamma}^{-1}T_{\Lambda}-M_{U,\Lambda, \Gamma}^{-1}T_{\Lambda}UT^{*}_{\Gamma}M^{-1}_{U,\Gamma,\Gamma}T_{\Gamma})\\
&=T_1\Big(M^{-1}_{U,\Gamma,\Gamma}T_{\Gamma}+T_2^{-1}\Psi(Id_{\ell^{2}(\oplus_{i}\mathcal{K}_{i} )}-UT^{*}_{\Gamma}M^{-1}_{U,\Gamma,\Gamma}T_{\Gamma})\Big)
\end{align*}
Conversely, suppose that $\Lambda$ is a g-frame for which
$$T_{\Lambda}=T_1\Big(M^{-1}_{U,\Gamma,\Gamma}T_{\Gamma}+
T_2^{-1}\Psi(Id_{\ell^{2}(\oplus_{i}\mathcal{K}_{i} )}-UT^{*}_{\Gamma}M^{-1}_{U,\Gamma,\Gamma}T_{\Gamma})\Big),$$
where $\Psi$ is an operator in $B(\ell^{2}(\oplus_{i}\mathcal{K}_{i}),\mathcal{H})$ and
$T_1,T_2\in B(\mathcal{H})$ are invertible operators. Then we have
\begin{align*}
M_{U,\Lambda,\Gamma}=T_{\Lambda}UT^{*}_{\Gamma}&=T_1\Big(M^{-1}_{U,\Gamma,\Gamma}
T_{\Gamma}+T_2^{-1}\Psi(Id_{\ell^{2}(\oplus_{i}\mathcal{K}_{i})}
-UT^{*}_{\Gamma}M^{-1}_{U,\Gamma,\Gamma}T_{\Gamma})\Big)UT^{*}_{\Gamma}\\
&=T_1\Big(M^{-1}_{U,\Gamma,\Gamma}T_{\Gamma}UT^{*}_{\Gamma}+T_2^{-1}\Psi
(UT^{*}_{\Gamma}-
UT^{*}_{\Gamma}M^{-1}_{U,\Gamma,\Gamma}T_{\Gamma}UT^{*}_{\Gamma})\Big)\\
&=T_1
\end{align*}
It follows that $ M_{U,\Lambda,\Gamma}$ is invertible.
\end{proof}


The proof of Theorem \ref{riesz} below which characterizes the invertibility of generalized g-Riesz multipliers relies on the following proposition.

\begin{proposition}\label{exce}
Let $\Lambda$ be a g-Bessel sequence for $\mathcal{H}$  with respect to $(\mathcal{K}_{i})_{i\in I}$ and let $U$ be a bounded operator on $\ell^2(\oplus_{i\in I}{\mathcal K}_i)$ which is also semi-normalized. Then the equality of the excess of g-frame $\Gamma$ with the excess of $\Lambda$, that is, $\dim(\ker(T_{\Gamma}))=\dim(\ker(T_{\Lambda}))$ is necessary for $\Gamma$ to participate to construct
invertible g-Bessel multipliers with g-frame $\Lambda$ and symbol $U$.
\end{proposition}
\begin{proof}
If we define  $V:\ell^{2}(\oplus\mathcal{K}_{i})\longrightarrow\ell^{2}(\oplus \mathcal{K}_{i})$ by
\begin{eqnarray*}
V\{x_{i}\}_{i\in I}:=(Id_{\ell^{2}(\oplus_i\mathcal{K}_{i})}-T^{*}_{\Gamma}M_{U, \Lambda,\Gamma}^{-1}T_{\Lambda}U)\{x_{i}\}_{i\in I},
\end{eqnarray*}
then it is not hard to check that
\begin{eqnarray*}
ran(V)=\ker (T_{\Lambda}U)
\end{eqnarray*}
On the other hand, using the equality
\begin{eqnarray*}
ran(T^{*}_{\Gamma}M_{U,\Lambda,\Gamma}^{-1})\oplus \ker((T^{*}_{\Gamma}M_{U, \Lambda,\Gamma}^{-1})^{*})=\ell^{2}(\oplus \mathcal{K}_{i})
\end{eqnarray*}
and the equality $T_{\Lambda}UT^{*}_{\Gamma}M_{U,\Lambda,\Gamma}^{-1}=Id_{\mathcal{H}}$  we see that
\begin{eqnarray*}
ran(V)=V(\ker((T^{*}_{\Gamma}M_{U, \Lambda,\Gamma}^{-1})^{*}))
\end{eqnarray*}
Hence,we have
\begin{align*}
\dim(\ker(T_{\Lambda}))&=\dim (U^{-1}(\ker(T_{\Lambda}))\\
&=\dim(\ker(T_{\Lambda}U))\\
&=\dim(V(\ker((T^{*}_{\Gamma}M_{U,\Lambda,\Gamma}^{-1})^{*}))\\
&\leq\dim(\ker((T^{*}_{\Gamma}M_{U,\Lambda,\Gamma}^{-1})^{*}))\\
&=\dim(\ker((M_{U,\Lambda,\Gamma}^{*})^{-1}T_{\Gamma})\\
&=\dim(\ker(T_{\Gamma})).
\end{align*}
Similarly, if we define $V':=Id_{\ell^{2}(\oplus_i \mathcal{K}_{i})}-U^{*}T^{*}_{\Lambda}(M_{U,\Lambda,\Gamma}^{*})^{-1}T_{\Gamma}$, then, using the equality
$$(M_{U,\Lambda,\Gamma}^{*})^{-1}T_{\Gamma}U^{*}T^{*}_{\Lambda}=Id_{\ell^{2}(\oplus_i \mathcal{K}_{i})},$$
one can show that
\begin{eqnarray*}
\dim(\ker(T_{\Gamma}))\leq \dim(\ker(T_{\Lambda}))
\end{eqnarray*}
We have now completed the proof of proposition.
\end{proof}


The following result completely characterizes the invertibility of generalized multiplier $M_{U,\Lambda,\Gamma}$
when one of the sequences is a g-Riesz basis.

\begin{theorem}\label{riesz}
Let $\Lambda$ be a g-Riesz basis for $\mathcal{H}$  with respect to $\{\mathcal{K}_{i}\}_{i\in I}$ and let $U$ be a bounded operator on $\ell^2(\oplus_{i\in I}{\mathcal K}_i)$. Then the following assertion hold.
\begin{enumerate}
\item If $U$ is semi-normalized, then $M_{U,\Lambda,\Gamma}$ is invertible if and only if $\Gamma$ is a g-Riesz basis.
\item If $\Gamma$ is a g-Riesz basis, then
\begin{enumerate}
\item the mapping $U\mapsto M_{U,\Lambda,\Gamma}$ from $B(\ell^{2}(\oplus_{i}\mathcal{K}_{i}))$ into $B({\mathcal H})$ is injective;
\item $M_{U,\Lambda,\Gamma}$ is invertible if and only if $U$ is semi-normalized;
\item if $U$ is semi-normalized, then $M_{U,\Lambda,\Gamma}^{-1}=M_{U^{-1},\widetilde{\Gamma},\widetilde{\Lambda}}$;
\item if $\{\xi_{j}\}_{j}$ refers to the orthonormal basis of $\ell^{2}(\oplus_{i}\mathcal{K}_{i})$ and
$K=\sup_{j}\|U(\xi_{j})\|$, then we have
$$K\sqrt{A_{\Lambda}A_{\Gamma}}\leq\|M_{U,\Lambda, \Gamma}\|\leq \sqrt{B_{\Lambda}B_{\Gamma}}\|U\|.$$
\end{enumerate}
\end{enumerate}
\end{theorem}
\begin{proof}
(1) First note that by Proposition \ref{exce} the invertibility of $M_{U,\Lambda,\Gamma}$ together with the fact that $U$ is semi-normalized, we have
$$\dim(\ker(T_{\Gamma}))=\dim(\ker(T_{\Lambda})).$$
From this, we can deduce that $\ker(T_\Gamma)$ is isomorphic to $\ker(T_\Lambda)=\{0\}$. It follows that the operator $T_\Gamma$ is injective and thus it is invertible. This means that $\Gamma$ is a g-Riesz basis. By biorthogonality of the sequences $\widetilde{\Lambda}$, $\Lambda$ and $\widetilde{\Gamma}$, $\Gamma$ the backward implication is trivial. In fact, $M_{U^{-1},\widetilde{\Gamma},\widetilde{\Lambda }}$ is the inverse of $M_{U,\Lambda,\Gamma}$.

To prove part (a) of (2), suppose that $M_{U_1,\Lambda,\Gamma}=M_{U_2,\Lambda,\Gamma}$. If $U:=U_1-U_2\neq 0$, then there exists $\{x_i\}_{i\in I}\in {\ell^{2}(\oplus_{i}\mathcal{K}_{i})}$ such that $U\{x_i\}_{i\in I}\neq 0$. We now invoke the surjectivity of the operator $T_\Gamma^*$ to conclude that there exists $x\in{\mathcal H}$ for which $T_\Gamma^*(x)=\{x_i\}_{i\in I}$. It follows that $UT_\Gamma^*(x)\neq 0$ whereas $T_\Lambda UT_\Gamma^*(x)=0$. Hence, we have
$$\big<T_\Lambda UT_\Gamma^*(x),y\big>=0\quad\quad\quad(y\in{\mathcal H}).$$
This means that $0\neq UT_\Gamma^*(x)\perp ran(T_\Lambda^*)={\mathcal H}$ which is a contradiction.

Now suppose that $M_{U,\Lambda,\Gamma}$ is invertible, then we observe that
\begin{align*}
UT^{*}_{\Gamma}M_{U,\Lambda,\Gamma}^{-1}T_{\Lambda}&=T^{*}_{\widetilde{\Lambda }}T_{\Lambda} UT^{*}_{\Gamma}M_{U,\Lambda,\Gamma}^{-1}T_{\Lambda}\\
&=T^{*}_{\widetilde{\Lambda }}M_{U,\Lambda,\Gamma}M_{U,\Lambda,\Gamma}^{-1}T_{\Lambda}\\
&=T^{*}_{\widetilde{\Lambda }}T_{\Lambda}\\
&=Id_{\ell^{2}(\oplus_{i}\mathcal{K}_{i})}.
\end{align*}
This together with part (1) and its proof proves parts (b) and (c).

Finally, in order to prove part (d) of (2), suppose that
$\xi_{i_{0}}$  is an arbitrary element of $\{\xi_{j}\}_{j}$. The surjectivity of the operator $T^{*}_{\Gamma}$ implies that there exists $x_{0}\in\mathcal{H}$ such that $T^{*}_{\Gamma}(x_{0})=\xi_{i_{0}}$. Hence, we have
\begin{eqnarray*}
\sqrt{A_{\Gamma}}\|x_{0}\|\leq\|T^{*}_{\Gamma}(x_{0})\|=\|\xi_{i_{0}}\|=1,
\end{eqnarray*}
and thus
\begin{align*}
\|M_{U,\Lambda,\Gamma}\|
=\sup_{x\in{\mathcal H}, x\neq 0}\dfrac{\|T_{\Lambda}UT^{*}_{\Gamma}(x)\|}{\|x\|}
\geq\dfrac{\|T_{\Lambda}UT^{*}_{\Gamma}(x_{0})\|}{\|x_{0}\|}
=\sqrt{A_{\Lambda}A_{\Gamma}}\|U(\xi_{i_{0}})\|.
\end{align*}
We have now completed the proof of proposition.
\end{proof}


We conclude this section by the following two results on the representation of the inverse of a generalized g-frame multiplier. The first one looks for a unique dual g-frame $\Gamma^\dag$ of $\Gamma$ such that for any dual g-frame $\Lambda^d$ of $\Lambda$ the inverse of $M_{U,\Lambda,\Gamma}$ can be represented using the
diagonal operator $U^{-1}$, $\Gamma^\dag$ and $\Lambda^d$, that is, $M_{U,\Lambda,\Gamma}^{-1}$ is again a generalized g-Bessel multiplier. The second one investigates invertible generalized g-frame multipliers $M_{U,\Lambda,\Gamma}$ whose inverses can be written as $M_{U^{-1},\widetilde{\Gamma},\widetilde{\Lambda}}$.

\begin{theorem}\label{uni-dual}
Let $ \Lambda $ and $ \Gamma $ be g-frames for $\mathcal{H}$  with respect to $\{\mathcal{K}_{i}\}_{i\in I}$ and let $U$ be semi-normalized. If $M_{U,\Lambda,\Gamma}$ is invertible, then there exists a unique dual g-frame $\Gamma^{\dagger}=\{\Gamma^{\dagger}_{i}\}_{i\in I}$ of $\Gamma$ such that
\begin{equation*}
M_{U,\Lambda,\Gamma}^{-1}=M_{U^{-1},\Gamma^{\dagger},\Lambda^{d}}.
\end{equation*}
for all dual g-frame $\Lambda^{d}$ of $\Lambda$.
\end{theorem}
\begin{proof}
The existence of $\Gamma^\dagger$ follows from the fact that
$$x=M_{U,\Lambda,\Gamma}^*(M_{U,\Lambda,\Gamma}^{-1})^*=\sum_{i\in I}\Gamma_i^*u_i^*\Lambda_i(M_{U,\Lambda,\Gamma}^{-1})^*.$$
In detail, $\Gamma^\dagger:=\{u_i^*\Lambda_i(M_{U,\Lambda,\Gamma}^{-1})^*\}_{i\in I}$
is a dual g-frame of $\Gamma$ for which $M_{U,\Lambda,\Gamma}^{-1}T_\Lambda=T_{\Gamma^\dagger}U^{-1}$ and we get
$$M_{U,\Lambda,\Gamma}^{-1}=M_{U,\Lambda,\Gamma}^{-1}T_\Lambda T^*_{\Lambda^d}=T_{\Gamma^\dagger}U^{-1}T^*_{\Lambda^d},$$
for all dual g-frame $\Lambda^d$ of $\Lambda$ as a consequence.
Let us now prove the uniqueness of $\Gamma^\dagger$,
which is the essential part of the theorem.
To this end, suppose that $\Gamma^\ddag$ is another dual g-frame of  $\Gamma$ for which $M_{U,\Lambda,\Gamma}^{-1}=M_{U^{-1},\Gamma^\ddag,\Lambda^{d}}$, then we would have the following equality
 \begin{eqnarray}\label{parsagelar}
T_{\Lambda^{d}}(U^{-1})^{*}T^{*}_{\Gamma^{\dagger}}=T_{\Lambda^{d}}
(U^{-1})^{*}T^{*}_{\Gamma^\ddag},
\end{eqnarray}
for all dual g-frame $\Lambda^d$ of $\Lambda$.
From this we deduce that  $(U^{-1})^{*}T^{*}_{\Gamma^{\dagger}}=(U^{-1})^{*}T^{*}_{\Gamma^\ddag}$ and thus the invertibility of $U$ implies that $\Gamma^\ddag=\Gamma^{\dagger}$.
Indeed, if for arbitrary $x$ in $\mathcal H$ we set
$$\{x_i\}_{i\in I}:=(U^{-1})^{*}T^{*}_{\Gamma^{\dagger}}(x)\quad\quad\quad{\hbox{and}}\quad\quad\quad
(U^{-1})^{*}T^{*}_{\Gamma^\ddag}(x):=\{y_i\}_{i\in I},$$
then equality \ref{parsagelar} implies that $T_{\Lambda^d}\{x_i-y_i\}_{i\in I}=0$
for all dual g-frame $\Lambda^d$ of $\Lambda$. Specially, for canonical dual $\widetilde{\Lambda}$ of $\Lambda$ we have
$$S_\Lambda^{-1}T_\Lambda\{x_i-y_i\}_{i\in I}=T_{\widetilde{\Lambda}}\{x_i-y_i\}_{i\in I}=0$$
and thus $T_\Lambda\{x_i-y_i\}_{i\in I}=0$.
If now for arbitrary $\{z_i\}_{i\in I}\in\ell^2(\oplus_i{\mathcal K}_i)$ and a fixed $e\in{\mathcal H}$ with $\|e\|=1$ we define
$$\Psi:{\mathcal H}\rightarrow\ell^2(\oplus_i{\mathcal K}_i);\quad x\mapsto\big<x,e\big>\{z_i\}_{i\in I},$$
then we observe that
\begin{align*}
\big<\{x_i-y_i\}_{i\in I},\{z_i\}_{i\in I}\big>&=\big<\{x_i-y_i\}_{i\in I},\Psi(e)\big>\\
&=\big<\Psi^*P_{\ker(T_\Lambda)}\{x_i-y_i\}_{i\in I},e\big>\\
&=\big<(T_{\widetilde{\Lambda}}+\Psi^*P_{\ker(T_\Lambda)})\{x_i-y_i\}_{i\in I},e\big>\\
&=\big<(T_{\Lambda_e^d}\{x_i-y_i\}_{i\in I},e\big>\\
&=0
\end{align*}
where $\Lambda_e^d:=\widetilde{\Lambda}_i+\pi_i\Psi^*P_{\ker(T_\Lambda)}$ is a dual g-frame of $\Lambda$. This means that $$\{x_i-y_i\}_{i\in I}\perp\ell^2(\oplus_i{\mathcal K}_i)$$
and therefore we should have $\{x_i\}_{i\in I}=\{y_i\}_{i\in I}$.
\end{proof}


Having reached this state it remains to find conditions guaranteeing the equality of $\Gamma^\dag$ with $\widetilde{\Gamma}$. This is the subject matter of the next result.

\begin{theorem}
Suppose that $\Lambda$ and $\Gamma$ are g-frames for $\mathcal{H}$ and that $U$ is semi-normalized. The following statements are equivalent.
\begin{enumerate}
\item $M_{U,\Lambda,\Gamma}$ is invertible and the unique g-frame $\Gamma^{\dagger}$ in Theorem \ref{uni-dual} is $\widetilde{\Gamma}$.
\item The optimal upper frame bound of the g-frame $\Gamma^{\dagger}$ is $\Big(\widetilde{A}_{\Gamma}\Big)^{-1}$.
\item $U\Gamma$ and $\Lambda$ are $Q$-equivalent g-frames.
\item The analysis operator of $U\Gamma$ obtains the
minimal norm, that is $$\|T^*_{U\Gamma}\|^2=\Big(\widetilde{A}_{\Gamma}\|Q^{-1}\|^2\Big)^{-1}.$$
\end{enumerate}
\end{theorem}
\begin{proof}
The implication (3)$\Leftrightarrow$(4) is proved in Proposition \ref{minimal}. Let us first prove that (1)$\Rightarrow$(2). To this end, suppose that $\Gamma^{\dagger}$ is the unique dual g-frames of $\Gamma$ for which
\begin{eqnarray*}
M^{-1}_{U,\Lambda,\Gamma}=M_{U^{-1},\Gamma^{\dagger},{\Lambda}^d},
\end{eqnarray*}
for all dual g-frame $\Lambda^{d}$ of $\Lambda$.
In light of \cite[Theorem 3.4]{aref} we have
\begin{eqnarray}\label{2.6}
\Gamma^{\dagger}_i=\widetilde{\Gamma}_i+\pi_{i}\Psi
\end{eqnarray}
for some $\Psi\in B({\mathcal H},\ell^2(\oplus_i{\mathcal K}_i))$ such that $T_\Gamma\Psi=0$, where, here and in the sequel, $\pi_i$ is the standard projection on the $i$-th component. Observe that
$T_{\Gamma^{\dagger}}=T_{\widetilde{\Gamma}}+\Psi^{*}$
and we get
$S_{\Gamma^{\dagger}}=S_{\widetilde{\Gamma}}+\Psi^{*}\Psi$. We now invoke the positivity of operators $S_{\Gamma^{\dagger}}$, $S_{\widetilde{\Gamma}}$ and $\Psi^{*}\Psi$ as well as the equality $S_{\widetilde{\Gamma}}=S_{\Gamma}^{-1}$ to conclude that
\begin{align}\label{2.7}
\|S_{\Gamma^{\dagger}}\|_{\rm op}&=\sup_{\|x\|=1}\langle\Gamma^{\dagger}x,x\rangle\nonumber\\
&=\|S_{\widetilde{\Gamma}}\|_{\rm op}+\sup_{\|x\|=1}\|\Psi x\|^{2}\nonumber\\
&=\|S_{\Gamma}^{-1}\|_{\rm op}+\sup_{\|x\|=1}\|\Psi x\|^{2}.
\end{align}
From this, by equality $\Gamma^\dagger=\widetilde{\Gamma}$, we deduce that $\Psi=0$. On the other hand, an argument similar to the proof of Proposition 5.4.4 of \cite{c} with the aid of \cite[Lemma 2.1]{Sun2} shows that $\|S_\Gamma^{-1}\|_{\rm op}=\Big(\widetilde{A}_{\Gamma}\Big)^{-1}$ and particularly $\|S_{\Gamma^\dagger}\|_{\rm op}=\widetilde{B}_{\Gamma^\dagger}$. It follows that $\widetilde{B}_{\Gamma^\dagger}=\Big(\widetilde{A}_{\Gamma}\Big)^{-1}$.

In order to prove that (2)$\Rightarrow$(3), first note that
Eq. (\ref{2.7}) together with the equality of the optimal upper frame bound of the g-frame $\Gamma^{\dagger}$ with $\Big(\widetilde{A}_{\Gamma}\Big)^{-1}$ imply that
$\Psi=0$. Hence, the equivalency of the g-frames $U\Gamma$ and $\Lambda$ follows from
Eq. (\ref{2.6}) and the definition of $\Gamma^\dagger$.

Finally, the proof will be completed by showing that (3)$\Rightarrow$(1). To do that, just noting that this is nothing more than routine calculations.
\end{proof}


\section{Some approaches
for constructing invertible generalized multipliers}

In this section, we present some
approaches for constructing of invertible generalized multipliers from a given one. In this respect, we
first recall the following perturbation condition from \cite{Naj}.


\begin{definition}
Let $ \Lambda =\lbrace\Lambda_{i}\rbrace_{i}$  be a sequence in  $ B(\mathcal{H}  , \mathcal{H}_{i})$ we say that a sequence $ \Lambda' =\lbrace \Lambda'_{i}\rbrace_{i}$
 in $ B(\mathcal{H} ,\mathcal{H}_{i})$ is a $ \mu $-perturbation of $ \Lambda $ if $ \parallel T_{\Lambda}-T_{\Lambda'}\parallel \leq \mu$.
 \end{definition}


\begin{theorem}\label{2.14}
Suppose that $\Lambda$ and $\Gamma$ are g-frames for $\mathcal{H}$ and suppose that $U$ is semi-normalized. For any  $\Lambda'$ which is a $\mu$-perturbation of $\Lambda$ with $\mu<\sqrt{A_{\Lambda}}$, there exists a Bessel sequence $\Gamma'$ which is a $\lambda\mu$-perturbation of $\Gamma$ for some $\lambda>0$ and $M_{U,\Lambda',\Gamma'}=M_{U,\Lambda,\Gamma}$.
\end{theorem}
\begin{proof}
First note that Theorem 3.5 of \cite{Naj} implies that $\Lambda'$ is a g-frame for $\mathcal{H}$ with lower bound $A_{\Lambda'}=(\sqrt{A_{\Lambda}}-\mu)^{2}$. Moreover, it is not hard to check that the sequence $U\Lambda$ is a g-frame for $\mathcal{H}$. Hence, we have
\begin{eqnarray*}
\ell^{2}(\oplus_{i}\mathcal{K}_{i})=\mathcal{R}(T^{*}_{U\Lambda})\oplus\ker (T_{U\Lambda})\hspace{1cm}and \hspace{1cm} T^{*}_{U\Lambda}T_{\widetilde{U\Lambda}}=T^{*}_{U\Lambda}S_{U\Lambda}^{-1}
T_{U\Lambda}=P_{\mathcal{R}(T^{*}_{U\Lambda})}.
\end{eqnarray*}
If we define
\begin{eqnarray*}
\Gamma':=\{u_j^{-1}\Lambda'_{i}S^{-1}_{U\Lambda}M_{U,\Lambda,\Gamma}+\pi_{i}P_{\ker (T_{U\Lambda})}T^{*}_{\Gamma}\}_{i\in I},
\end{eqnarray*}
then we have
\begin{align*}
T_{\Gamma'}=T_{\Gamma}U^{*}T^{*}_{\Lambda}T_{\widetilde{\Lambda'}}(U^{*})^{-1}
+T_{\Gamma}P_{\ker(T_{U\Lambda})}.
\end{align*}
Hence, we observe that
\begin{align}\label{9}
T_{\Gamma'}-T_{\Gamma}&=T_{\Gamma}U^{*}T^{*}_{\Lambda}T_{\widetilde{\Lambda'}}
(U^{*})^{-1}+T_{\Gamma}P_{\ker (T_{U\Lambda})} \nonumber \\
&-T_{\Gamma}(P_{\mathcal{R}(T^{*}_{U\Lambda})}+P_{\ker (T_{U\Lambda})})\nonumber\\
&=T_{\Gamma}U^{*}T^{*}_{\Lambda}T_{\widetilde{\Lambda'}}(U^{*})^{-1}-T_{\Gamma}
P_{\mathcal{R}(T^{*}_{U\Lambda})}\nonumber\\
&=T_{\Gamma}U^{*}(T^{*}_{\Lambda}-T^{*}_{\Lambda'})T_{\widetilde{\Lambda'}}(U^{*})^{-1}.
\end{align}
On the other hand, by the Open Mapping Theorem, one can conclude that there are constants
$a,b>0$ such that $a<\|U\|<b $ and thus
\begin{align*}
\|T_{\Gamma'}-T_{\Gamma}\|_{\rm op}&\leq \|T_{\Gamma}\|_{\rm op}\|U^{*}\|_{\rm op} \|(T^{*}_{\Lambda}-T^{*}_{\Lambda'})\|_{\rm op} \|T_{\widetilde{\Lambda'}}\|_{\rm op} \|(U^{*})^{-1}\|_{\rm op}\\
&\leq\dfrac{b\sqrt{B_{\Gamma}}\mu }{a(\sqrt{A_{\Lambda}}-\mu)}.
\end{align*}
If we set $ \lambda=\dfrac{b\sqrt{B_{\Gamma}}}{a(\sqrt{A_{\Lambda}}-\mu)} $, then $ \Gamma' $ is a $ \lambda\mu $-perturbation of $\Gamma$  and particularly
\begin{align*}
M_{U,\Lambda',\Gamma'}&=T_{\Lambda'}UT^{*}_{\Gamma'}\\
&=T_{\Lambda'}U(U^{-1}T^{*}_{\widetilde{\Lambda'}}T_{\Lambda}UT^{*}_{\Gamma})+
T_{\Lambda'}U(P_{\ker T_{\Lambda'}U}T_{\Gamma})\\
&=T_{\Lambda}UT^{*}_{\Gamma}\\&=M_{U,\Lambda,\Gamma}.
\end{align*}
This completes the proof.
\end{proof}


The next result shows that in Theorem \ref{2.14} above
our choice of $\Gamma'$ turns out to be perfect in terms of best approximations with respect to the norm $\|{\mathcal X}\|_{\mathcal{F}}:=\|T_{\mathcal X}\|_{\rm op} $, where ${\mathcal X}=\{{\mathcal X}_{i}\}_{i\in I}$ is in $g{\mathcal F}({\mathcal H})$,
the set of all g-frames in  $\mathcal{H}$, whenever $ M_{U,\Lambda,\Gamma} $ is invertibe g-frame multiplier.

\begin{theorem}
Let $\Lambda$ and $\Gamma$ be g-frames for $\mathcal{H}$ and let $U$ be semi-normalized symbol such that $M_{U,\Lambda,\Gamma}$ is invertible. If $\Lambda'$ is a $\mu$-perturbation of $\Lambda$ with $\mu<\sqrt{A_{\Lambda}}$, Then
\begin{eqnarray*}
\Gamma':=\{u_i^{-1}\Lambda'_{i}S^{-1}_{U\Lambda}M_{U,\Lambda,\Gamma}+\pi_{i}P_{\ker (T_{U\Lambda})}T^{*}_{\Gamma}\}_{i\in I}
\end{eqnarray*}
is a best approximation of $\Gamma$ in $g{\mathcal F}({\mathcal H})$ and a $\lambda\mu$-perturbation of $\Gamma$ such that $M_{U,\Lambda',\Gamma'}=M_{U,\Lambda,\Gamma} $, where $\lambda=\frac{b\sqrt{B_{\Gamma}}}{a\sqrt{A_{\Lambda}}-\mu}$.
\end{theorem}
\begin{proof}
According to the theorem \ref{2.14} and its  proof, it suffices to show that $\Gamma'$ is the best approximation of $\Gamma$. To this end, suppose that $\Gamma^{''}$ is another g-frame for $\mathcal{H}$ such that $M_{U,\Lambda',\Gamma^{''}}= M_{U,\Lambda,\Gamma}$. For each $x\in \mathcal{H}$, we have
\begin{align*}
x=M_{U,\Lambda',\Gamma^{''}}( M_{U,\Lambda,\Gamma} ^{-1}x)
=T_{U\Lambda}T^{*}_{\Gamma^{''}M_{U,\Lambda,\Gamma} ^{-1}}x.
\end{align*}
Hence, $\Gamma^{''}M_{U,\Lambda,\Gamma} ^{-1}$ is a dual g-frame for $U\Lambda$ and thus there exists a bounded operator $\Psi:\mathcal{H}\longrightarrow \ell^2(\oplus_{i}\mathcal{K}_{i})$ such that
\begin{eqnarray*}
\Gamma^{''}_iM_{U,\Lambda,\Gamma} ^{-1}=u_i^*\Lambda_{i}S^{-1}_{U\Lambda}+\pi_{i}\Psi
\end{eqnarray*}
and $T_{U\Lambda}\Psi=0$. Hence, the equality $\ell^{2}(\oplus_{i}\mathcal{K}_{i})=\mathcal{R}(T^{*}_{U\Lambda})\oplus \ker (T_{U\Lambda})$, implies that
\begin{eqnarray*}
\Psi^{*}\{x_{i}\}_{i\in I}=\Psi^{*}(P_{\mathcal{R}(T^{*}_{U\Lambda})}\{ x_{i}\}_{i\in I}+P_{\ker(T_{U\Lambda})}\{x_{i}\}_{i\in I})=\Psi^{*}P_{\ker (T_{U\Lambda})}\{ x_{i}\}_{i\in I},
\end{eqnarray*}
for all $\{x_{i}\}_{i\in I}\in \ell^{2}(\oplus_{i}\mathcal{K}_{i})$. Moreover, we see that
\begin{align}\label{1.12}
T_{\Gamma}U^{*}(T^{*}_{\Lambda}-T^{*}_{\Lambda'})T_{\widetilde{U\Lambda}}&=
T_{\Gamma}U^{*}(T^{*}_{\Lambda}-T^{*}_{\Lambda'})S^{-1}T_{U\Lambda} \nonumber\\
&=T_{\Gamma}U^{*}(T^{*}_{\Lambda}-T^{*}_{\Lambda'})S^{-1}T_{U\Lambda}( P_{\mathcal{R}(T^{*}_{U\Lambda})}+P_{\ker (T_{\theta})})\nonumber\\
&=T_{\Gamma}U^{*}(T^{*}_{\Lambda}-T^{*}_{\Lambda'})S^{-1}T_{U\Lambda}
P_{\mathcal{R}(T^{*}_{U\Lambda})}\nonumber\\
&=T_{\Gamma}U^{*}(T^{*}_{\Lambda}-T^{*}_{\Lambda'})T_{\widetilde{U\Lambda}}
P_{\mathcal{R}(T^{*}_{U\Lambda})}.
\end{align}
Now by equality
\begin{eqnarray*}
\Gamma^{''}_{i}=(u_i^*\Lambda_{i}S^{-1}_{U\Lambda}+\pi_{i}\Psi)M_{U,\Lambda,\Gamma} \hspace{1cm}(i\in I)
\end{eqnarray*}
we have
\begin{align*}
T_{\Gamma^{''}}\{x_{i}\}&=\sum_{i\in I}(\Gamma^{''}_{i})^{*}(x_{i})\nonumber\\
&=\sum_{i\in I}M^{*}_{U,\Lambda,\Gamma}(S^{-1}_{U\Lambda}\Lambda^{*}_{i}u_i+\Psi^{*}\pi_{i}^{*})\{ x_{i}\}\nonumber\\
&=M_{U^{*},\Gamma,\Lambda }(S^{-1}_{U\Lambda}T_{U\Lambda}\{x_{i}\}+\Psi^{*}\{ x_{i}\})\nonumber\\
&=M_{U^{*},\Gamma,\Lambda }(T_{\widetilde{U\Lambda}}\{x_{i}\}+\Psi^{*}\{x_{i}\})
\end{align*}
for each $\{x_{i}\}_{i\in I}\in \ell^{2}(\oplus_{i}\mathcal{K}_{i})$. Hence, the equation $ T^{*}_{U\Lambda}T_{\widetilde{U\Lambda}}=P_{\mathcal{R}(T^{*}_{\theta})}$ conclude that
\begin{align*}
T_{\Gamma^{''}}-T_{\Gamma}&=M_{U^{*},\Gamma,\Lambda }(T_{\widetilde{U\Lambda}}+\Psi^{*})-T_{\Gamma}(P_{\mathcal{R}
(T^{*}_{\widetilde{U\Lambda}})}+P_{\ker (T_{\widetilde{U\Lambda}})})\\
&=M_{U^{*},\Gamma,\Lambda }T_{\widetilde{U\Lambda}}-T_{\Gamma}
P_{\mathcal{R}(T^{*}_{\widetilde{U\Lambda}})}+M_{U^{*},\Gamma,\Lambda }\Psi^{*}-T_{\Gamma}P_{\ker (T_{\widetilde{U\Lambda}})}\\
&=T_{\Gamma}U^{*}T^{*}_{\Lambda}T_{\widetilde{U\Lambda}}-
T_{\Gamma}T^{*}_{U\Lambda}T_{\widetilde{U\Lambda}}+M_{U^{*},\Gamma,\Lambda }\Psi^{*}-T_{\Gamma}P_{\ker (T_{\widetilde{U\Lambda}})}\\
&=T_{\Gamma}U^{*}T^{*}_{\Lambda}T_{\widetilde{U\Lambda}}-
T_{\Gamma}U^{*}T^{*}_{\Lambda'}T_{\widetilde{U\Lambda}}+M_{U^{*},\Gamma,\Lambda }\Psi^{*}-T_{\Gamma}P_{\ker (T_{\widetilde{U\Lambda}})}\\
&=T_{\Gamma}U^{*}(T^{*}_{\Lambda}-T^{*}_{\Lambda'})T_{\widetilde{U\Lambda}}
+M_{U^{*},\Gamma,\Lambda }\Psi^{*}P_{\ker (T_{\widetilde{U\Lambda}})}-T_{\Gamma}P_{\ker (T_{\widetilde{U\Lambda}})}\\
&=T_{\Gamma}U^{*}(T^{*}_{\Lambda}-T^{*}_{\Lambda'})
T_{\widetilde{U\Lambda}}P_{\mathcal{R}(T^{*}_{\widetilde{U\Lambda}})}+
(M_{U^{*},\Gamma,\Lambda }V^{*}-T_{\Gamma})P_{\ker (T_{\widetilde{U\Lambda}})}.
\end{align*}
We now invoke the equality $\mathcal{R}(T^{*}_{\widetilde{U\Lambda}})=\ker (T_{\widetilde{U\Lambda}})^{\perp}$ to deduce that
\begin{eqnarray*}
\|(T_{\Gamma^{''}}-T_{\Gamma})\{x_{i}\}\|^{2}\geq \| T_{\Gamma}U^{*}(T^{*}_{\Lambda}-T^{*}_{\Lambda'})
T_{\widetilde{U\Lambda}}P_{\mathcal{R}(T^{*}_{\widetilde{U\Lambda}})}\{ x_{i}\}\|^{2}
\end{eqnarray*}
for all $\{x_{i}\}_{i}\in \ell^{2}(\oplus_{i}\mathcal{K}_{i})$. By using equations \eqref{9} and \eqref{1.12}, we have
\begin{eqnarray*}
\|(T_{\Gamma^{''}}-T_{\Gamma})\|_{\rm op}^{2}\geq  \| (T_{\Gamma}-T_{\Gamma'})\|^{2}_{\rm op}.
\end{eqnarray*}
Therefore $\|(T_{\Gamma^{''}}-T_{\Gamma})\|_{\mathcal{F}}\geq  \| (T_{\Gamma}-T_{\Gamma'})\|_{\mathcal{F}} $ and the proof is completed.
\end{proof}


We conclude this paper with the following necessary condition for invertibility of generalized g-Bessel multiplier which is quite different
from those studied in the previous literatures on this topic.

\begin{theorem}
Let $\Lambda$ be a g-frame for $\mathcal{H}$ and $\Lambda^{d}$ be the dual g-frame of $\Lambda$. Assume that $\Gamma$ is a sequence of bounded operators from $ \mathcal{H}$ into ${\mathcal K}_i$ ($i\in I$) which satisfies
the following  two conditions:
\begin{enumerate}
\item $\lambda:=\sum_{i\in I}\parallel \Lambda_{i}-\Gamma_{i} \parallel^{2}<\infty$;
\item $\mu:=\sum_{i\in I}\parallel \Lambda_{i}-u_i^{*}\Gamma_{i} \parallel \parallel \Lambda_{i}^{d}\parallel<1$.
\end{enumerate}
If $M$ denote the operator $ M_{U,\Gamma,\Lambda^{d}}$, then $\Gamma$ is a g-frame for $\mathcal{H}$ with bounds $\frac{1}{\|U\|^{2}B_{\Lambda^d}}(1-\mu)^{2}$ and $B_\Lambda(1+\frac{\sqrt{\lambda}}{\sqrt{B_\Lambda}})^{2} $ and $M$ is invertible on $\mathcal{H}$ and for all $x\in \mathcal{H}$
\begin{eqnarray*}
\frac{1}{1+\mu}\parallel x\parallel\leq\parallel M^{-1} x\parallel\leq\frac{1}{1-\mu}\parallel x\parallel.
\end{eqnarray*}
\end{theorem}
\begin{proof}
We first show that $ \Gamma $ is a g-Bessel sequence.
To do that, we define
 \begin{eqnarray*}
T_{\Gamma}:\ell^2(\oplus_{i}\mathcal{K}_{i})\longrightarrow {\mathcal H},\hspace{.5cm} T_{\Gamma}(\lbrace x_{i}\rbrace_{i\in I})=\sum_{i\in I}\Gamma^{*}_{i}(x_{i}).
\end{eqnarray*}
The condition (1) implies that $T_{\Gamma}$ is well defined and $\|T_{\Gamma}\|_{\rm op}\leq \sqrt{\lambda}+\sqrt{B}$.
Moreover, by condition (2), for each $x\in \mathcal{H}$, we have
\begin{align*}
\|x-M_{U,\Gamma,\Lambda^{d}}x\|&=\|\sum_{i\in I}(\Lambda^{*}_{i}-\Gamma_{i}^{*}u_{i})( \Lambda_{i}^{d}(x))\|\\
&\leq\sum_{i\in I}\|\Lambda^{*}_{i}-\Gamma_{i}^{*}u_{i}\| \|\Lambda_{i}^{d}(x)\| \|x\|\\
&=\mu \|x\|.
\end{align*}
This means that $M_{U,\Gamma,\Lambda^{d}}$ is invertible and $\frac{1}{1+\mu}
\leq\|M^{-1}\|\leq\frac{1}{1-\mu}$.
Particularly, the invertibility of $M$ implies that every $x\in \mathcal{H}$ can be written as
\begin{eqnarray*}
x=MM^{-1}x=\sum_{i\in I}\Gamma_{i}^{*}u_{i}(\Lambda_{i}^{d}(M^{-1}x)).
\end{eqnarray*}
This leads to
\begin{align*}
\|x\|^{2}=\langle x,x\rangle &= \langle \sum_{i\in I}\Gamma_{i}^{*}u_{i}(\Lambda_{i}^{d}(M^{-1}x)),x\rangle\\
&\leq\sum_{i\in I}\langle \Lambda^{d}_{i}(M^{-1}x),u_{i}^{*}\Gamma_{i}x\rangle\\
&\leq\left(\sum_{i\in I} \|\Lambda^{d}_{i}(M^{-1}x)\|^{2} \right) ^{{\frac{1}{2}}}\left(\sum_{i\in I} \|u_{i}^{*}\Gamma_{i}x\|^{2}\right)^{\frac{1}{2}}\\
&\leq\sqrt{D}\|M^{-1}x\|\left(\|\sum_{i\in I}u_{i}^{*}\Gamma_{i}x\|^{2}\right)^{\frac{1}{2}}\\
&\leq\sqrt{D} \frac{1}{1-\mu}\|x\|\|U^{*}T^{*}_{\Gamma}x\|\\
&\leq\sqrt{D} \frac{1}{1-\mu}\|x\|\|U^{*}\|\|T^{*}_{\Gamma}x\|.
\end{align*}
Therefore,
\begin{eqnarray*}
\frac{1}{\parallel U\parallel D}(1-\mu)^{2}\parallel f\parallel ^{2}\leq  \parallel T^{*}_{\Gamma}f\parallel
\end{eqnarray*}
This completes the proof.
\end{proof}



\begin{thebibliography}{9}


\bibitem{aref} { A. A. Arefijamaal and S. Ghasemi}, On
characterization and stability of alternate dual of g-frames, {\it Turk. J. Math.}
{\bf 37} (2013), 71--79.


\bibitem{balaz3} { P. Balazs},  Basic definition and properties of Bessel
multipliers, {\it J. Math. Anal. Appl.} {\bf 325} (2007),
571--85.


\bibitem{continu}
P. Balazs, D. Bayer and A. Rahimi, Multipliers for continuous frames
in Hilbert spaces, {\it J. Phys. A} {\bf 45} (2012), 244023, 20 pp.



\bibitem{balaz11} {P. Balazs and D. T. Stoeva}, Representation of
the inverse of a frame multiplier, {\it J. Math. Anal. Appl.} {\bf
422} (2015), 981--994.



\bibitem{c} { O. Christensen,} An Introduction to Frames
and Riesz Bases, Birkh\"{a}user, (2016).


\bibitem{gene}
H. Hosseinnezhad, Gh. Abbaspour Tabadkan and A. Rahimi, g-frames and their generalized multipliers in Hilbert spaces, {\it Ann. Funct. Anal.} {\bf 10} (2019), 180--195.


\bibitem{japp} {H. Javanshiri},
Some properties of approximately dual frames in Hilbert spaces,
{\it Results. Math.} {\bf 70} (2016), 475--485.


\bibitem{JNF}
H. Javanshiri, Invariances of the operator properties of frame multipliers under perturbations of frames and symbol, {\it Numer. Funct. Anal. Optim.} {\bf 39} (2018), 571--587.


\bibitem{JLAA}
H. Javanshiri and M. Choubin, Multipliers for von Neumann–Schatten Bessel sequences in separable Banach spaces, {\it Linear Algebra Appl.} {\bf 545} (2018), 108--138.



\bibitem{jatom} {H. Javanshiri, A. Fattahi}, Continuous atomic systems for subspaces,
{\it Mediterr. J. Math.} {\bf 13} (2016), 1871--1884.





\bibitem{Naj}
A. Najati, M.H. Faroughi and A. Rahimi, { G-frames and stability of g-frames in
Hilbert spaces}, {\it Methods Funct. Anal. Topology}, {\bf 14} (2008), 271--286.



\bibitem{rahim1}
A. Rahimi, Multipliers of generalized frames in Hilbert spaces, {\it Bull.
Iranian Math. Soc.} {\bf 37} (2011), 63--80.



\bibitem{Rahi}
A. Rahimi and P. Balazs, { Multipliers for p-Bessel sequences in Banach spaces}, {\it Integral Equations Operator Theory.} {\bf 68} (2010), 193--205.




\bibitem{Stoeva}
D. T. Stoeva and P. Balazs, {Invertibility of multipliers}, {\it Appl. Comput. Harmon. Anal.}
{\bf 33} (2012), 292--299



\bibitem{Sun}
W. Sun, {Stability of g-frames}, {\it J. Math. Anal. Appl.} {\bf 236} (2007), 858--868



\bibitem{Sun2}
W. Sun, {G-frames and G-Riesz bases}, {\it J. Math. Anal. Appl.} {\bf 322} (2006), 437--452


\end{thebibliography}
\end{document}